\RequirePackage{fix-cm}
\documentclass[smallcondensed]{svjour3}     % onecolumn (ditto)
\smartqed  % flush right qed marks, e.g. at end of proof
\usepackage{amssymb,amsmath}
\usepackage{pstricks,pst-math,pst-infixplot}
\usepackage{pst-plot}
\usepackage{graphicx}
\usepackage[misc,geometry]{ifsym}
\usepackage[latin5]{inputenc}
\usepackage{pst-node}
\usepackage{pstricks-add}
\DeclareMathOperator*{\argmin}{arg\,min}
\usepackage{graphicx}

\newcommand{\rn}{\mathbb{R}^{n}}

\newcommand{\R}{\mathbb{R}}

\newcommand{\ut}{u_{\theta}}

\begin{document}

\title{On Reduction of Exhausters via a Support Function Representation%\thanks{Grants or other notes
%about the article that should go on the front page should be
%placed here. General acknowledgments should be placed at the end of the article.}
}

%\titlerunning{Short form of title}        % if too long for running head

\author{Didem Tozkan}

%\authorrunning{Short form of author list} % if too long for running head

\institute{D. Tozkan \at
              Eski\c{s}ehir Technical University,  Department of Mathematics, 26470, Eski\c{s}ehir, Turkey \\
              Tel.:+90-222-3213550
              \email{dtokaslan@eskisehir.edu.tr\\
               ORCID ID: 0000-0001-5244-5544}
}

\date{Received: date / Accepted: date}
% The correct dates will be entered by the editor

\maketitle

\begin{abstract}
Exhausters are families of compact, convex sets which provide minmax or maxmin representations of positively homogeneous functions and they are efficient tools for the study of nonsmooth functions \cite{d2}.  Upper and lower exhausters of positively homogeneous functions are employed to describe optimality conditions in geometric terms and also to find directions of steepest descent or ascent. Since an upper/lower exhauster may contain finitely or infinitely many compact convex sets, the problem of minimality and reduction of exhausters naturally arise. There are several approaches to reduce exhausters \cite{ab1,gfinite,ros2,ros3,reduction}. In this study, in the sense of inclusion-minimality, some reduction techniques for upper exhausters of positively homogeneous functions defined from  $\R^2$  to $\R$  is proposed by means of a representation of support functions. These techniques have concrete geometric meanings and they form a basis for a necessary and sufficient condition for inclusion-minimality of exhausters. Some examples are presented to illustrate each reduction technique.
% \PACS{PACS code1 \and PACS code2 \and more}
\end{abstract}

\section{Introduction}
\label{intro}
The concept of exhausters which gives a newpoint of view to optimality is defined by
Demyanov \cite{d2}. The exhaustive families of upper convex and lower concave approximations,
defined by Pschenichnyi \cite{psch}, were used to define exhausters.
Exhausters are families of convex bodies providing minmax or maxmin representations of positively homogeneous functions and they are effective tools for the study of nonsmooth functions \cite{ab2,d1,dr3,dr4}. These notions are employed to describe optimality conditions in geometric terms relying on the fact that the directional derivative and the generalized ones (such as Dini, Hadamard, Clarke and Michel-Penot derivatives) are positively homogeneous functions of directions \cite{d1} and they also can be used to find directions of steepest descent or ascent.
Demyanov and Roshchina gave some optimality conditions in terms of exhausters and
they obtained some relationships between exhausters and some generalized subdifferentials
\cite{dr3,dr4}. Furthermore, some optimality conditions in terms of lower and upper exhausters were introduced in \cite{ab2,dr1,dr2,weakexh}.

Since an upper/lower exhauster of a positively homogeneous function may contain
infinitely many compact convex sets, reduction of exhausters is an important part
of this theory. Reduction techniques mean to obtain minimal exhausters or smaller exhausters by inclusion or by form. Roshchina defined minimality of exhausters both by inclusion and by form and gave some reduction techniques to reduce exhausters \cite{ros2,ros3}. Grzybowski et al. \cite{gfinite} gave a criterion, in terms of shadowing sets, for an arbitrary upper exhauster to be an exhauster of sublinear function and a criterion for the minimality of finite upper exhausters. Also, they showed that minimal exhausters do not have to be unique.
On the other hand, K\" u\c c\" uk et al. defined weak exhausters consisting convex sets corresponding to the weak subgradients of a directional derivative of a function, and they gave some conditions to reduce weak exhausters \cite{weakexh,reduction}. Recently, Abbasov \cite{ab1} propose new conditions for the verification of minimality of exhausters and present some techniques for their reduction.

In classical mathematics the most widely used representation scheme for convex bodies is the support function representation \cite{bf,schn}. It was introduced by Minkowski in 1903 \cite{mink}, and has been extensively studied by mathematicians thereafter. Ghosh and Kumar \cite{gosh} showed that the support function representation of convex bodies can be very effectively used in computing variety of geometric operations within a single framework. The idea of a single framework is to establish that such geometric operations are nothing but simple algebraic transformations of the support functions of the operand objects. They presented that the support function can be viewed not as a single representation, but one of a class of representation schemes. Since the support function is a real-valued function, arithmetic operations such as addition, subtraction, reciprocal and max-min of support functions give rise to geometric operations such as Minkowski addition (dilation), Minkowski decomposition (erosion), polar duality, and union-intersection of corresponding convex bodies, respectively.

In this work, we propose a different approach to reduce exhausters using support function representations of convex bodies. The framework of this study is limited to the functions defined on $\R^2$. First we show that an upper exhauster $E^*$ of $h$ can be reduced to a smaller by inclusion one by examining the support function representations $\rho_C$ of sets $C\in E^*$. Since the support function representation of a convex body occurs as a piecewise function of sinusoidal curves, we examine the contribution of each $\rho_C$ to the pointwise minimum value of support functions, namely $\rho:=\underset{C\in E^*}{\min}{\rho_C}$. It is proved that a convex body can be discarded from $E^*$ if it does not contribute to the value of $\rho$. Furthermore we present some reduction techniques which have apparent geometric meanings to obtain an exhauster smaller by inclusion. Also, we obtain a necessary and sufficient condition for a convex body not to be removed from an upper exhauster. Furthermore we characterize inclusion minimality of an upper exhauster by using support function representations of convex bodies. The method presented in this work is based on the support function representations of the convex bodies and when these representations are displayed together for an exhauster, it can be clearly revealed which of the sets  are necessary or unnecessary. Therefore among the other methods given so far for the reduction of exhausters, our new approach is much more convenient to make interpretations geometrically and decide to discard a set from an exhauster .

The paper is organized as follows. In Sect. 2 we give some basic definitions and summarize the concept of support function representation of convex bodies. In Sect. 3 we present the main results on reducing exhausters and we introduce a necessary and sufficient condition for the inclusion minimality of an upper exhauster. We illustrate each of the results and demonstrate the usage of these reduction techniques. In Sect. 4 we give a discussion of presented results.

\section{Preliminaries}

Let us recall some basic notions, properties and state the notations that we used throughout this study.
\begin{definition}\cite{rock}
A function $h:\rn\to\R$ is called a \textit{positively homogeneous function (p.h.)} of degree one if
\[h(\lambda g)=\lambda h(g),\ \ \ \forall\lambda\geq0,\ \forall g\in\rn.\]
\end{definition}

\begin{definition}\cite{rock}
Let $A\subseteq\rn$ be a nonempty set and $x\in\rn$. The function defined as
\begin{equation}\label{Eq1}
H(A,x)=\underset{a\in A}{\sup}{\langle a,x\rangle}
\end{equation}
is called the \textit{support function} of the set $A$.
\end{definition}

As mentioned in \cite{gosh}, since the support function of a set $A$ is a p.h. function, it is more convenient to use the function $H(A,u)$ where $u\in S^{n-1}$ denotes a unit vector of $\rn$. $H(A,u)$ is a complete representation of the set $A$ because the values of $H(A,u)$ for all $u\in S^{n-1}$ completely determine $A$ as
\begin{equation}
A=\{x\in\rn\ |\ \langle x,u\rangle\leq H(A,u),\ \forall u\in S^{n-1}\}.
\end{equation}
That means $A$ is the intersection of all the half-spaces $\langle x,u\rangle\leq H(A,u)$. Moreover if $A$ is a convex body (i.e., nonempty, compact, convex set), $H(A,u)$ can be obtained by (\ref{Eq1}) considering only the boundary points of $A$ instead of using every point $a\in A$, that is
\begin{equation}\label{Eq2}
H(A,u)=\underset{a\in bdA}{\max}{\langle a,u\rangle}, \ \ \forall u\in S^{n-1}
\end{equation}
where $bdA$ denotes the boundary of the set $A$.
\begin{definition}\cite{d2}
Let $h:\rn\to\R$ be a \textup{p.h.} function.
\begin{enumerate}
\item[\textup{(a)}] A family $E^*$ of nonempty, compact and convex sets in $\rn$ is called an \textit{upper
exhauster} of $h$ if
\begin{equation*}
h(g)=\inf_{C\in E^*}\max_{v\in C}\langle v,g\rangle \text{ for all }g\in \R^n.
\end{equation*}
\item[\textup{(b)}] A family $E_*$ of nonempty, compact and convex sets in $\rn$ is called a \textit{lower
exhauster} of $h$ if
\begin{equation*}
h(g)=\sup_{C\in E_*}\min_{w\in C}\langle w,g\rangle \text{ for all }g\in \R^n.
\end{equation*}
\end{enumerate}
\end{definition}

The minimality of exhausters both by inclusion and by form was defined in \cite{ros2,ros3}. In order to check optimality conditions in terms of exhausters or find directions of steepest descent or ascent, one prefer to deal with rather "smaller exhausters" that has two meanings: Smaller  by means of quantity of sets or smaller by means of the size (breadth, width) of sets. Hence, the aim of reduction techniques is to obtain smaller or (if possible) minimal exhausters by inclusion or by form.
  \begin{definition}\cite{ab1}
    Let $h$ be a p.h. function, $E_1$ and $E_2$ be upper (lower) exhausters of $h$. If $E_1\subset E_2$, then $E_1$ is called smaller by inclusion than $E_2$.
    \end{definition}
    \begin{definition}\cite{ros3}
    An upper (lower) exhauster $E$ of the p.h. function $h$ is called minimal by inclusion, if there is no other upper (lower) exhauster $\widetilde{E}$ which is smaller by inclusion than $E$.
    \end{definition}

 \begin{definition}\cite{ab1}
  An upper (lower) exhauster $E_1$ of the p.h. function $h$ is said to be smaller by form than other upper (lower) exhauster $E_2$ of $h$, if for all $\widetilde{C}\in E_1$ there exists a set $C\in E_2$ such that $\widetilde{C}\subset C$.
    \end{definition}
  \begin{definition}\cite{ros3}
    An upper (lower) exhauster $E$ of the p.h. function $h$ is called minimal by form, if there is no other smaller by form upper (lower) exhauster $\widetilde{E}$ of $h$.
    \end{definition}
 Note that an exhauster minimal by form is also minimal by inclusion, but the converse is not true \cite{ros3}.

Here we summarize the concept of support function representation of convex bodies to be used in this work. The reader can find further properties in \cite{gosh}. Let $A\subseteq\R^2$ be a convex polygon. The set $A$ can be represented in terms of its support function as a sequence of sinusoidal curves in another $2$-dimensional space namely $\theta\rho$-space.
For each edge of $A$ there is only one supporting line, and if the outer normal direction of an edge $e_i$ is $u_i$, then the equation of the corresponding supporting line $L(A,u_i)$ consists of vectors $X=(x,y)$ such that $\langle X,u_i\rangle=H(A,u_i)$. We use the notation $H(A,u_i)=\rho_i$, shortly.

On the other hand, in 2-dimensional Euclidean space a unit vector $u\in S^1$ is uniquely determined by the angle $\theta$ (in radians) between $u$ and the positive $x$-axis, i.e., $u=(\cos\theta,\sin\theta)$. Therefore, the equation of supporting line $L(A,u_i)$ of $A$ at the direction $u_i=(\cos\theta_i,\sin\theta_i)$ can be expressed as
\[x\cos\theta_i+y\sin\theta_i=\rho_i.\]

If we consider a new coordinate system having $\theta$ and $\rho$ values as its axes, the line $L(A,u_i)$, and hence the edge $e_i$ can be presented as a point in $\theta\rho$-space having the coordinate $(\theta_i,\rho_i)$ (See e.g. Figure \ref{fig1}).

% Şekil 1
\begin{figure}[h]
\begin{center}
\includegraphics[scale=.5]{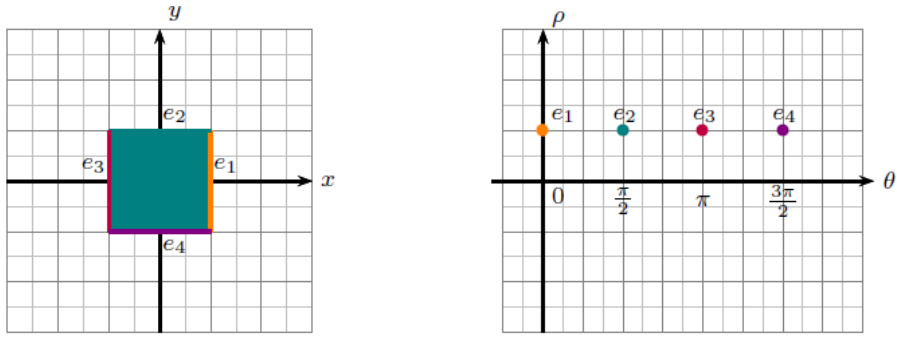}
\end{center}
\caption{The set $A=[-1,1]\times[-1,1]$ and representations of its edges $e_1,e_2,e_3,e_4$ as points in $\theta\rho$-space.}\label{fig1}
\end{figure}

Moreover, a point in $xy$-space can be transformed into a sinusoidal curve in $\theta\rho$-space.
Let $v_i=(x_i,y_i)$ is a vertex of polygon $A$.
The support function of the singleton $\{v_i\}$ is clearly $H({\{v_i\}},u)=\langle v_i,u\rangle$.
Using the notations $H({\{v_i\}},u)=\rho$ and $u=(\cos\theta,\sin\theta)$ we can write the $\theta\rho$-representation of vertex $v_i$
\begin{equation}\label{Eq3}
\rho=x_i\cos\theta+y_i\sin\theta=\Lambda\sin(\theta+\phi)
\end{equation}
where $\Lambda=\sqrt{x_i^2+y_i^2}$ and $\phi=\tan^{-1}(x_i/y_i)$. Equation (\ref{Eq3}) is the representation of the point $(x_i,y_i)$ in the $\theta\rho$-space and it is clearly a sinusoidal curve.

To sum up, a convex polygon $A$ in the $xy$-space can be transformed in the $\theta\rho$-space into a sequence of sinusoidal curves representing its vertices and the intersection point between two consecutive sine curves representing the respective edges of $A$ (e.g., Fig. 2).
In the same way, we may consider the representations of a circle, ellipse, etc. in the $\theta\rho$-spaces.

%Şekil2
\begin{figure}[h]
\begin{center}
\includegraphics[scale=.5]{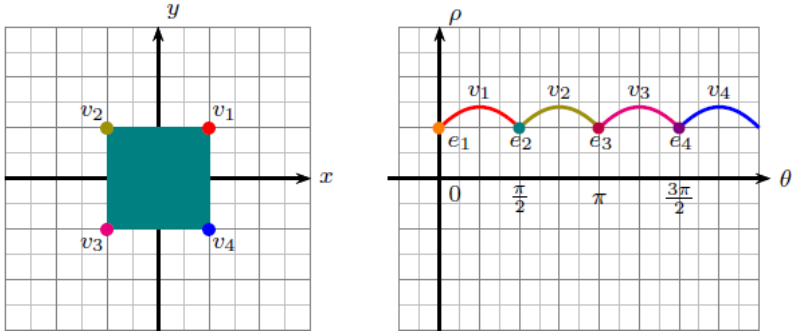}
\end{center}
\caption{The set $A=[-1,1]\times[-1,1]$ and representations of its vertices $v_1,v_2,v_3,v_4$ as sinusoidal curves in $\theta\rho$-space.}
\label{fig2}
\end{figure}

Geometric operations such that convex hull, intersection, Minkowski addition, Minkowski decomposition, polar duality etc. on convex bodies corresponds to algebraic operations such that Max-Min, addition, subtraction, reciprocal etc. on support functions of convex bodies, respectively. Examinations of these geometric operations by means of support functions were given in \cite{gosh}, in detail.

If $h$ is a p.h. Lipschitz function, then for all $u\in S^{n-1}$ it can be represented as
\[h(u)=\underset{C\in E^*}{\min}{H(C,u)}=\min_{C\in E^*}\max_{v\in C}\langle v,u\rangle\]
where $E^*$ is an upper exhauster of $h$ (see \cite{cas,d1}).
On the other hand, min operation of support functions corresponds to intersection operation of sets in $E^*$ (see \cite{gosh}, page 387). But the support function of the intersection of these convex bodies may not be the pointwise minimum of these support functions. Namely, even if the intersection $\widehat C:=\bigcap_{C\in E^*} C$ is nonempty, the value $H(\widehat C,u)$ may not be equal to the value of $h(u)=\underset{C\in E^*}{\min}{H(C,u)}$. For example if we consider the p.h. function $h(u_1,u_2)=\min\{\max\{u_1+2u_2,2u_1+4u_2,3u_1+2u_2\},\max\{u_1+u_2,u_1+3u_2,3u_1+3u_2,3u_1+u_2\}\}$ then it is clear that $E^*=\{C_1,C_2\}$ is an upper exhauster of $h$ where $C_1=conv\{(1,2),(2,4),(3,2)\}$ and $C_2=conv\{(1,1),(1,3),(3,1),(3,3)\}$. Taking the point $u=(1,1)$ we see that $h(1,1)=6$, but $H(C_1\cap C_2, (1,1))=\frac{11}{2}$. Thus there exists a $u\in\mathbb R^2$ such that $h(u)=\min\{H(C_1,u),H(C_2,u)\}> H(C_1\cap C_2,u)$. Hence we conclude that it is not suitable to intersect all of the sets in an exhauster to reduce it. Therefore, it is complicated to eliminate the sets of which support functions do not contribute to the value of $h(u)$. So we need an approach that takes into account the structure of an exhauster to reduce it adequately.

\section{Some Reduction Techniques for Upper Exhausters}

\indent \indent In this section, in the sense of inclusion-minimality, some reduction techniques for upper exhausters of p.h. functions defined on  $\R^2$ are proposed by means of support function representations.

\indent \indent These techniques lead us to decide whether a set can be discarded or not by observing the support function representations (they are generally sequences of sinusoidal curves in $\theta\rho$-space) of all sets belonging to an exhauster.
They also have very clear geometric meanings and interpretations allowing us to obtain (if possible) an exhauster that is smaller by inclusion. In addition, we present a necessary and sufficient condition for inclusion-minimality of exhausters.

Throughout this work $h:\R^2\to\R$ is a p.h. function and $E^*$ is an upper exhauster of $h$ that provides the representation
\[h(u)=\underset{C\in E^*}{\min}{H(C,u)}, \ \ \text{for all } u\in S^1.\]

Each of the convex body $C\in E^*$ can be represented by a curve uniquely determined as follows. Let $C\in E^*$ be an arbitrary convex body. It is clear from (\ref{Eq2}) that for every direction $u\in S^1$ there exists a boundary element of $v_u\in bd C$ satisfying $H(C,u)=\langle v_u,u\rangle$. If we indicate the direction vector $u$ by the angle $\theta$ between $u$ and the $x-$axis, we represent each of the direction vectors $u$ of $S^1$ with an angle $\theta\in[0,2\pi]$ where $u=\ut=(\cos\theta,\sin\theta)$. Therefore for all $\theta\in[0,2\pi]$ there exists a boundary element $v_{\theta}=(x_{\theta},y_{\theta})$ of $C$ such that
\begin{equation}
H(C,u_{\theta})=\langle v_{\theta},u_{\theta}\rangle=x_{\theta}\cos\theta+y_{\theta}\sin\theta.
\end{equation}
Hence, we define the function $\rho_C:[0,2\pi]\to\R$ as
\begin{equation}
\rho_C(\theta):=x_{\theta}\cos\theta+y_{\theta}\sin\theta
\end{equation}
which is simply the $\theta\rho$-representation of $C$. This function is well-defined because the inner product $\langle v_{\theta},u_{\theta}\rangle=x_{\theta}\cos\theta+y_{\theta}\sin\theta$ have the same value for all of the boundary points $v_{\theta}$ (corresponding to angle $\theta$) which intersect the supporting line $L(C,u_{\theta})$.
\begin{example}
Consider the set $C=[1,2]\times[1,2]$ and the direction $u_0=(1,0)$ (equi\-valently $\theta=0$). It is easy to see that the equation of supporting line to $C$ in the direction $u_0$ is $x=2$. Hence the intersection of the supporting line with $C$ is the line segment $conv\{(2,1),(2,2)\}$ and $\rho_C(0)=x_0 \cos0+y_0\sin0 =x_0=2$ for all $(x_0,y_0)\in conv\{(2,1),(2,2)\}$.
\end{example}

%Şekil3
\begin{figure}[h]\begin{center}
\includegraphics[scale=.5]{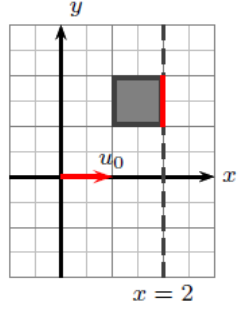}
\end{center}
\caption{Example 1}\label{fig3}
\end{figure}

For simplicity of notation we define the function
\begin{equation}\rho(\theta):=\underset{C\in E^*}{\min}{\rho_C(\theta)}=\min_{C\in E^*}\max_{v\in C}\langle v_{\theta},u_{\theta}\rangle
\end{equation}
and it is obvious that $\rho(\theta)=h(u_{\theta})$.

Let us now express the results of how the sets in an upper exhauster can be reduced using $\theta\rho$-representations.
Firstly, we give the main theorem that leads to obtain practical reduction techniques in the sense of inclusion-minimality.
\begin{theorem}\label{Theo1}
Let $C_0$ be an arbitrary set of $E^*$. Then, $\bar E^*=E^*\setminus\{C_0\}$ is also an upper exhauster of $h$ (which means $C_0$ can be discarded from $E^*$) if and only if for all $\theta\in[0,2\pi]$ there exists a set $C_{\theta}\in \bar E^*$ such that
\begin{equation*}
\rho_{C_{\theta}}(\theta)\leq\rho_{C_0}(\theta)
\end{equation*}
holds.
\end{theorem}
\begin{proof}
Let $\bar E^*=E^*\setminus\{C_0\}$ be also upper exhauster of $h$ which means
\[\rho(\theta)=\underset{C\in \bar E^*}{\min}{\rho_C(\theta)}.\]
By the definition of the function $\rho$ we see that $\rho(\theta)\leq\rho_{C_0}(\theta)$ for all $\theta\in[0,2\pi]$. Hence
\begin{equation}\label{Eq4}
\underset{C\in \bar E^*}{\min}{\rho_C(\theta)}\leq\rho_{C_0}(\theta), \forall\theta\in[0,2\pi].
\end{equation}
On the other hand for an arbitrary $\theta\in[0,2\pi]$ there exists a set $C_{\theta}\in \bar E^*$ such that
\begin{equation}\label{Eq5}
\underset{C\in \bar E^*}{\min}{\rho_C(\theta)}=\rho_{C_{\theta}}(\theta).
\end{equation}
Thus from (\ref{Eq4}) and (\ref{Eq5}) we find $C_{\theta}\in \bar E^*$ satisfying $\rho_{C_{\theta}}(\theta)\leq\rho_{C_0}(\theta)$ for all $\theta\in[0,2\pi]$.

Conversely, assume that for all $\theta\in[0,2\pi]$ there exists a set $C_{\theta}\in \bar E^*$ such that
\begin{equation}\label{Eq6}
\rho_{C_{\theta}}(\theta)\leq\rho_{C_0}(\theta)
\end{equation} To prove that $\bar E^*$ is also an upper exhauster of $h$ we need to show that \linebreak $\rho(\theta)=\underset{C\in \bar E^*}{\min}{\rho_C(\theta)}$ for all $\theta\in[0,2\pi]$. For an arbitrary $\theta\in[0,2\pi]$ we consider the set $C_{\theta}$ satisfying (\ref{Eq6}) by hypothesis. Since this $C_{\theta}$ is an element of $\bar E^*$ we have
\begin{equation*}
\underset{C\in \bar E^*}{\min}{\rho_C(\theta)}\leq \rho_{C_{\theta}}(\theta)\leq \rho_{C_0}(\theta).
\end{equation*}
Hence we obtain
\begin{equation*}
\rho(\theta)=\min\{\underset{C\in \bar E^*}{\min}{\rho_C(\theta)},\rho_{C_0}(\theta)\}=\underset{C\in \bar E^*}{\min}{\rho_C(\theta)}
\end{equation*}
which completes the proof.
\end{proof}
In order to obtain smaller exhausters by inclusion we present various techniques in the light of Theorem \ref{Theo1}.
\begin{corollary}\label{cor1}
For a set $C_0\in E^*$ if
\[\rho_{C_0}(\theta)>\rho(\theta),\ \text{for all }\theta\in[0,2\pi]\]
is satisfied, then $\bar E^*=E^*\setminus\{C_0\}$ is also an upper exhauster of $h$, that means $C_0$ can be discarded from $E^*$.
\end{corollary}
\begin{proof}
Let $\theta\in[0,2\pi]$ be an arbitrary direction. Since $\rho(\theta)=\underset{C\in E^*}{\min}{\left\{\underset{v\in C}{\max}{\langle v,u_{\theta}\rangle}\right\}}$ there exists a set $C_{\theta}$ that produces the minimum value, i.e.,
\begin{equation*}
C_{\theta}:=\underset{C\in E^*}{\argmin}{\left\{\underset{v\in C}{\max}{\langle v,u_{\theta}\rangle}\right\}}.
\end{equation*}
Then by hypothesis we have
\begin{equation*}
\rho_{C_0}(\theta)>\rho(\theta)=\rho_{C_{\theta}}(\theta)
\end{equation*}
and it is clear that $C_{\theta}\neq C_0$. Therefore for all $\theta\in[0,2\pi]$ we find a set $C_{\theta}\in \bar E^*$ such that
\begin{equation*}
\rho_{C_{\theta}}(\theta)<\rho_{C_0}(\theta)
\end{equation*} which means $\bar E^*$ is also an upper exhauster of $h$, by Theorem \ref{Theo1}.
\end{proof}

\begin{example}
Consider the function
$$\begin{array}{l}
h(x,y)=\min\{\max\{x+y,-x+y,0\};
 x+2y; \max\{-2x+y,-2x+2y\}; -x-y; \max\{-x+y,-x+2y\}\ \}.
\end{array}$$
The family $E^*=\{C_0,C_1,C_2,C_3,C_4\}$ is an upper exhauster of $h$ where \linebreak $C_0=conv\{(0,0),(1,1),(-1,1)\}, C_1=\{(1,2)\}, C_2=conv\{(-2,1),(-2,2)\}$, \linebreak $C_3=\{(-1,1)\}$ and $C_4=conv\{(-1,1),(-1,2)\}$ (see Fig. \ref{fig4}(a)). It is clear that
\[\rho_{C_0}(\theta)>\rho(\theta),\ \text{for all }\theta\in[0,2\pi]\]
as seen in Fig. \ref{fig4}(b)-(c). Hence, $C_0$ can be discarded from $E^*$ by Corollary \ref{cor1}.
\end{example}

%Şekil4
\begin{figure}[h]
\begin{tabular}{rcl}
\includegraphics[scale=.5]{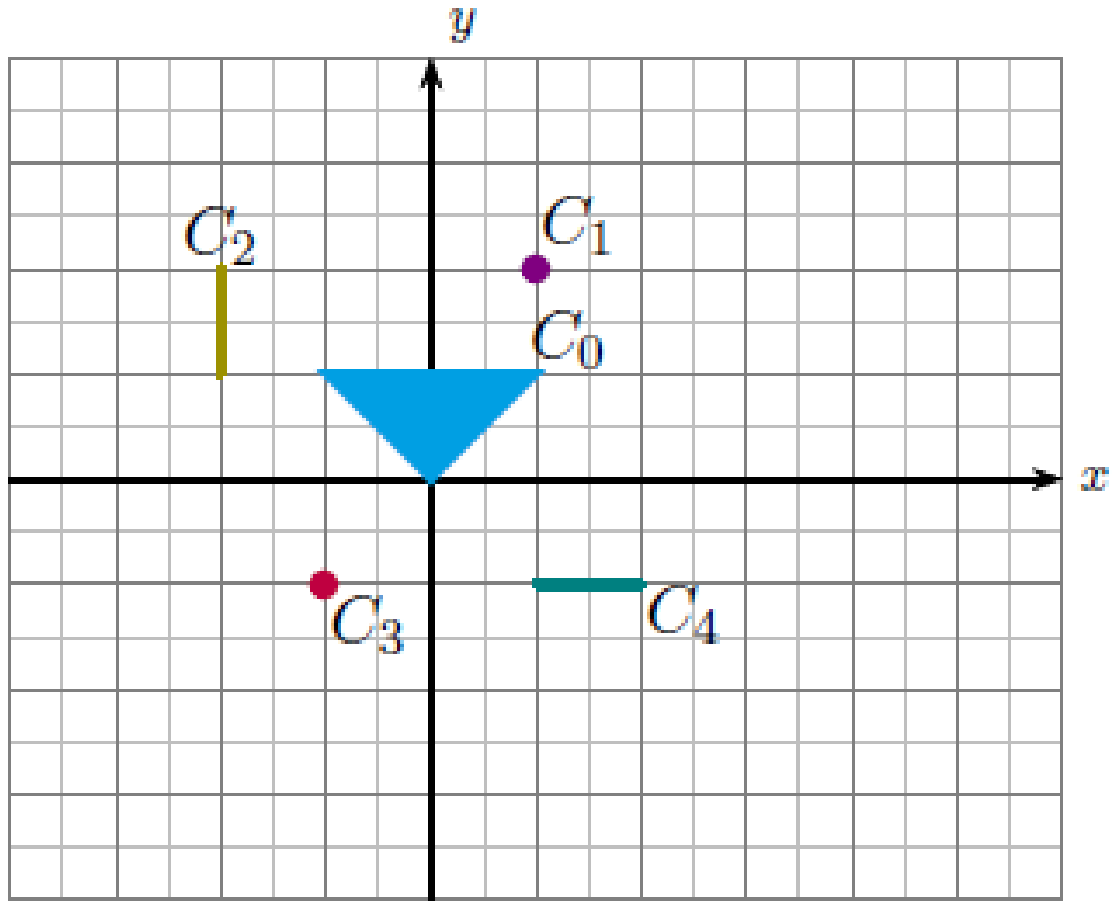}
&
\hspace{-.4cm}
\includegraphics[scale=.5]{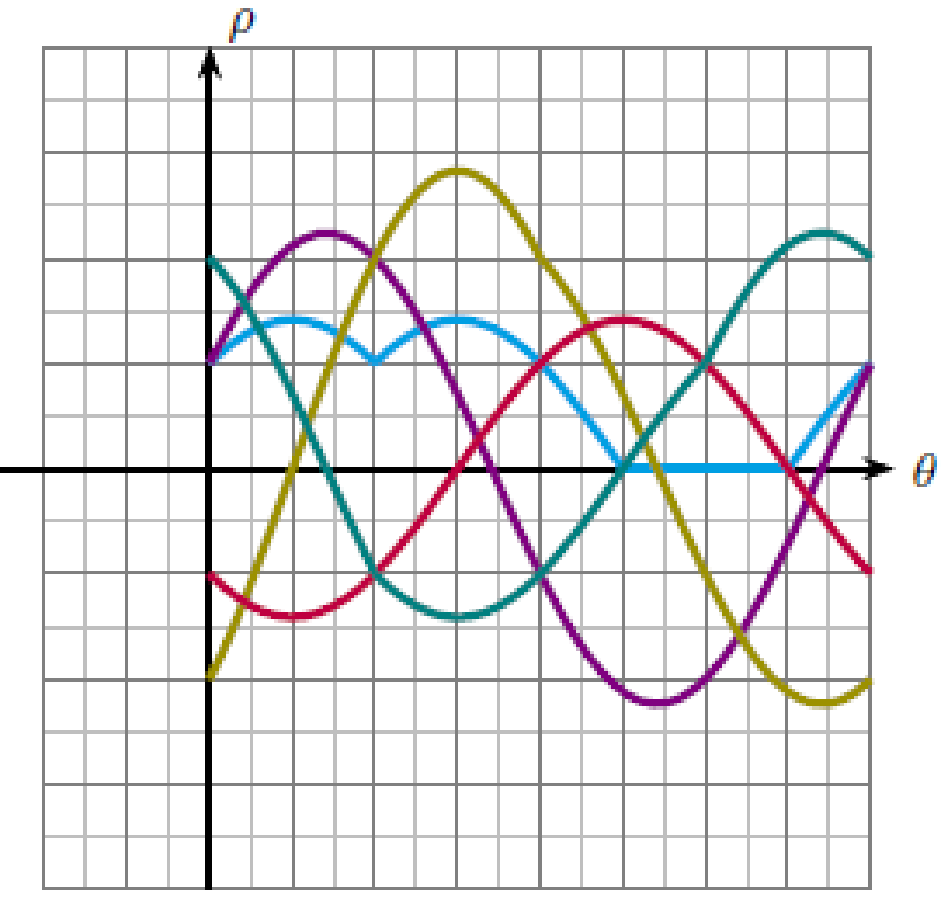}
&
\hspace{-.4cm}
\includegraphics[scale=.5]{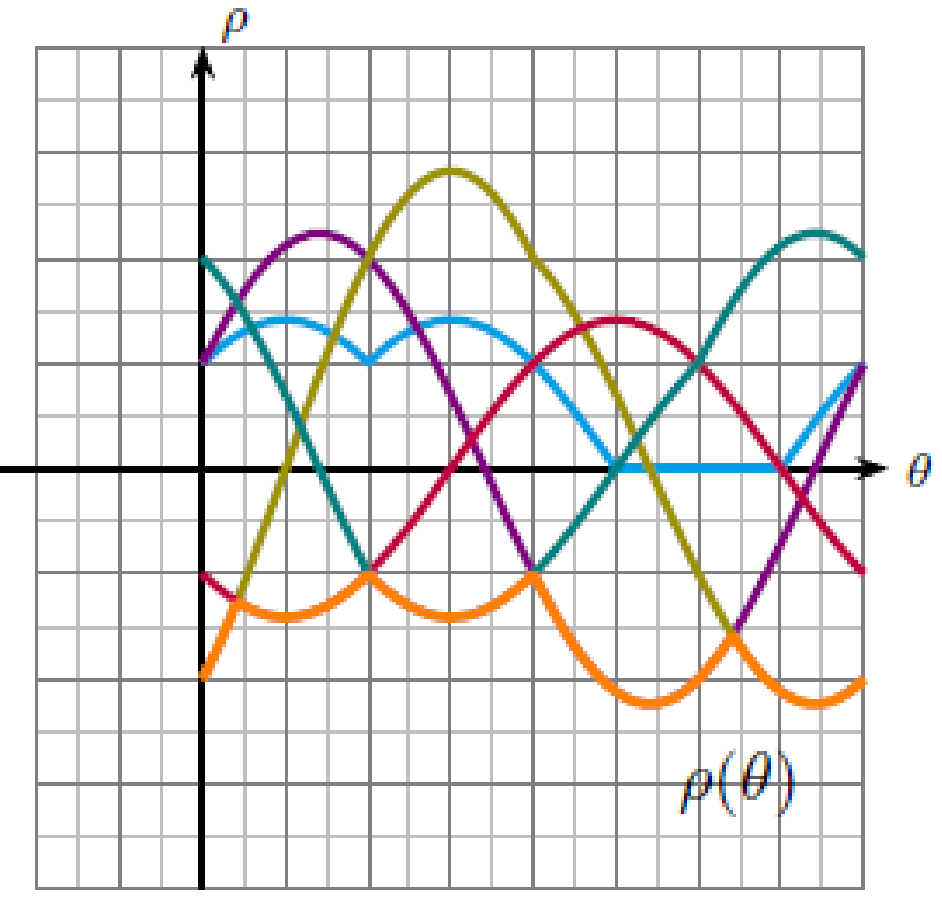}
\end{tabular}
\caption{(a) The sets belonging to the upper exhauster $E^*$ of Example 2. (b) $\theta\rho$-representations of the sets $C_i$ for $i=0,\ldots,4$. (c) $\rho_{C_0}(\theta)>\rho(\theta),\ \text{for all }\theta\in[0,2\pi]$.}\label{fig4}
\end{figure}

\begin{corollary}\label{Sonuc1}
For a set $C_0\in E^*$ if there exists a set $\tilde{C}\in\bar E^*$ such that
\begin{equation}\label{Eq7}
\rho_{\tilde C}(\theta)\leq\rho_{C_0}(\theta),\ \text{for all } \theta\in[0,2\pi]
\end{equation}
holds, then $\bar E^*=E^*\setminus\{C_0\}$ is also an upper exhauster of $h$.
\end{corollary}
\begin{proof}
If we take $C_{\theta}$ as $\tilde C$ for all $\theta\in[0,2\pi]$, then $C_0$ can be discarded from $E^*$ by Theorem \ref{Theo1}.
\end{proof}
\begin{remark}
The inequality (\ref{Eq7}) is equivalent to $H(\tilde C,u_{\theta})\leq H(C_0,u_{\theta})$ for all $u_{\theta}\in S^1$ which means $\tilde C\subseteq C_0$. Therefore, Corollary \ref{Sonuc1} means that any superset of a set $C_0\in E^*$ can be removed from the exhauster. This result was stated in \cite[Theorem 4.3(i)]{ros2} and here we restate it in a different way.
\end{remark}

%% BURADA KALDIM
\begin{corollary}\label{cor3}
Let $C_0\in E^*$, $n\in\mathbb N^+$ and $\mathcal P=\{\alpha_0,\alpha_1,\ldots,\alpha_n\}$ be a partition of $[0,2\pi]$ where $0=\alpha_0<\alpha_1<\ldots<\alpha_n=2\pi$. Consider the family of intervals $\mathcal B=\{B_j=[\alpha_j,\alpha_{j+1}]\ |\ j=0,1,\ldots,n-1\}$ generated by $\mathcal P$. For all $B_j\in\mathcal B$ if there exists a set $C_j\in \bar E^*=E^*\setminus \{C_0\}$ satisfying
\[\rho_{C_j}(\theta)\leq\rho_{C_0}(\theta),\ \text{for all }\theta\in B_j\]
then $C_0$ can be discarded from $E^*$, hence $\bar E^*$ is also an upper exhauster of $h$.
\end{corollary}

\begin{proof}
Take an arbitrary $\theta\in[0,2\pi]$. Since $\mathcal P$ be a partition of $[0,2\pi]$, then there exists $B_j^{\theta}\in\mathcal B$ such that $\theta\in B_j^{\theta}$. By the hypothesis, for this $B_j^{\theta}$ there exists a set $C_j^{\theta}\in\bar E^*$ satisfying
\[\rho_{C_j^{\theta}}(\theta)\leq \rho_{C_0}(\theta), \ \forall \theta\in B_j^{\theta}.\]
Thus if we consider $C_j^{\theta}$ as $C_{\theta}$ in Theorem \ref{Theo1}, it is clear that $C_0$ can be discarded from $E^*$.
\end{proof}
\begin{example}
Let us consider the function
$$\begin{array}{l}
h(x,y)=\min\{\max\{0,2x,2y,2x+2y\};\max\{0,-x,-y\};-2x-2y\}
\end{array}
$$
The family $E^*=\{C_0,C_1,C_2\}$ is an upper exhauster of $h$ where
$$\begin{array}{c}
 C_0=conv\{(0,0),(-1,0),(0,-1)\},  C_1=\{(-2,-2)\}, C_2=conv\{(0,0),(2,0),(0,2),(2,2)\}
\end{array}$$
\end{example}

%Şekil5
\begin{figure}
\begin{tabular}{ccc}
\includegraphics[scale=.5]{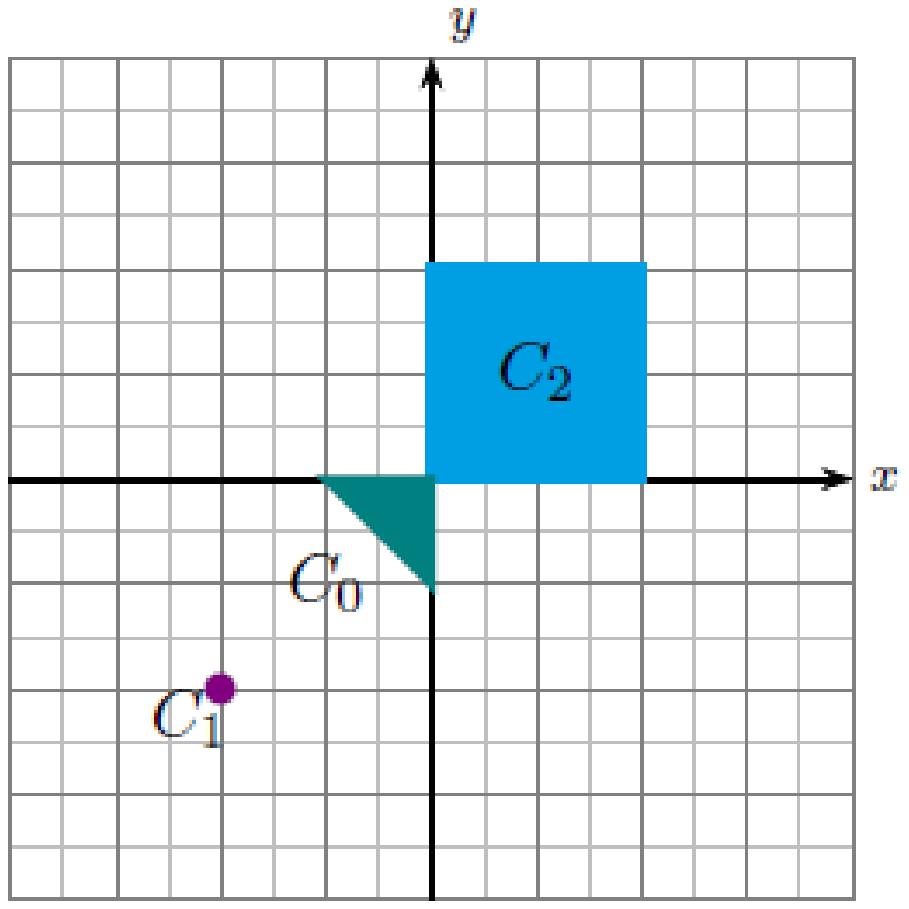}
&
\hspace{-.4cm}
\includegraphics[scale=.5]{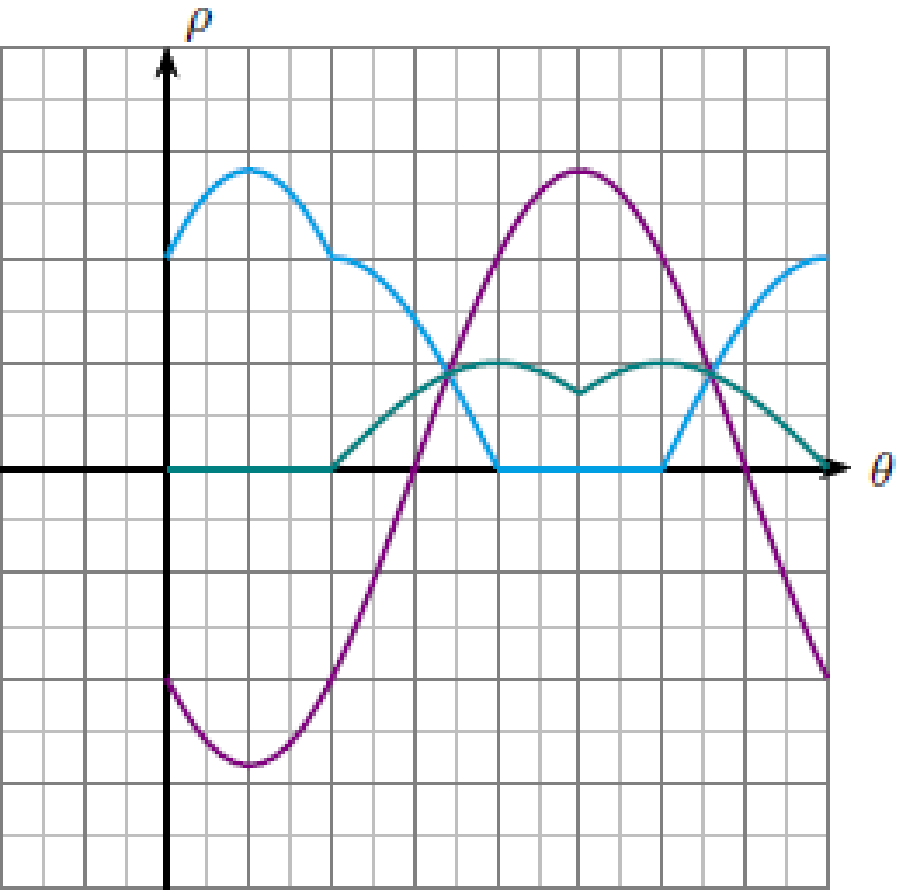}
&
\hspace{-.4cm}
\includegraphics[scale=.5]{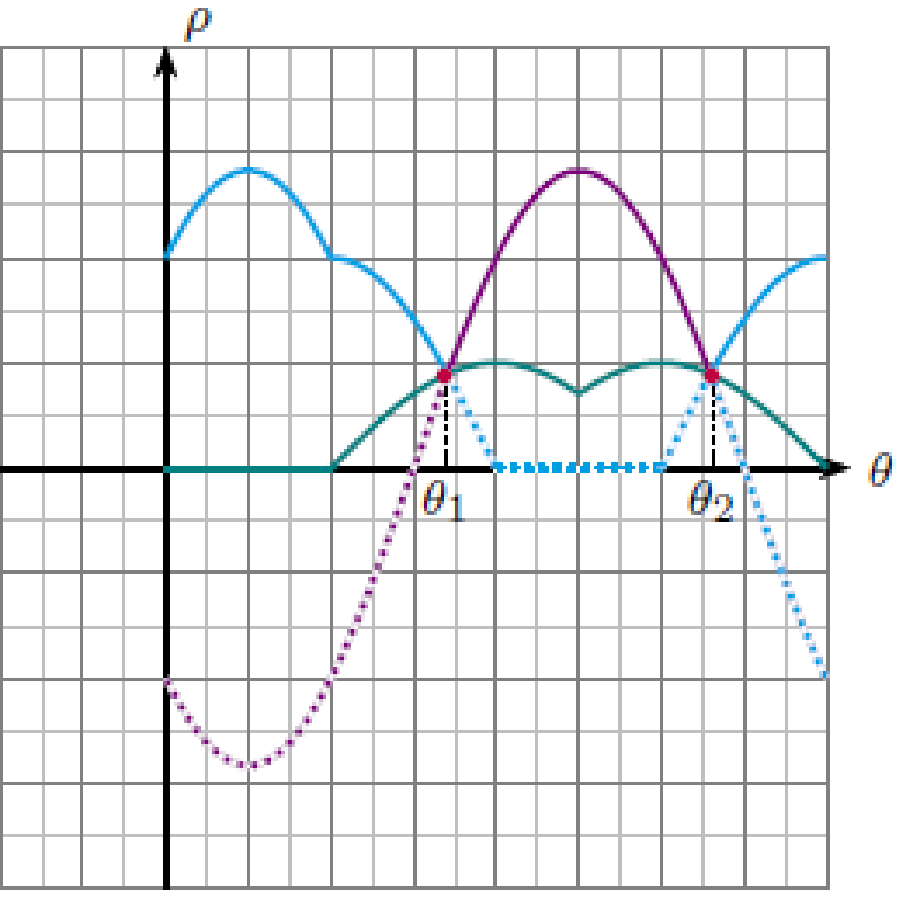}
\end{tabular}
\caption{(a) The sets belonging to the upper exhauster $E^*$ of Example 3. (b) $\theta\rho$-representations of the sets $C_0, C_1$ and $C_2$. (c) The set $C_0$ can be discarded by Corollary \ref{cor3}.}\label{fig5}
\end{figure}

(see Fig.  \ref{fig5}(a)).  One can obtain with some basic calculations that
$$\begin{array}{ll}
\rho_{C_1}(\theta)\leq\rho_{C_0}(\theta) & ,\forall \theta\in[0,\theta_1]\\
\rho_{C_2}(\theta)\leq\rho_{C_0}(\theta)& ,\forall \theta\in[\theta_1,\theta_2]\\
\rho_{C_1}(\theta)\leq\rho_{C_0}(\theta)& ,\forall \theta\in[\theta_2,2\pi]\\
\end{array}$$
where $\theta_1\approx2,68\pi,\ \theta_2\approx5.17\pi$ and $\mathcal B=\{[0,\theta_1], [\theta_1,\theta_2],[\theta_2,2\pi]\}$ is a partition of $[0,2\pi]$ as seen in Fig. \ref{fig5}(b)-(c).
Therefore, we can discard the set $C_0$ by Corollary \ref{cor3} that means  $\bar E^*=\{C_1,C_2\}$ is an upper exhauster of $h$ which is smaller by inclusion.
\newpage
In the following theorem, we present a sufficient condition for the impossibility of discarding a set from an upper exhauster.
\begin{theorem}\label{teo2}
Let $h$ be p.h. function, $E^*$ be an upper exhauster of $h$, $C_0\in E^*$ and assume that $E^*$ does not contain any proper subset of $C_0$. If there exist some $\theta_1, \ \theta_2\in[0,2\pi]$, $\theta_1<\theta_2$ such that
\[\rho_{C_0}(\theta)=\rho(\theta),\ \text{for all }\theta\in[\theta_1,\theta_2]\]
then $\bar E^*=E^*\setminus \{C_0\}$ is not an upper exhauster of $h$.
\end{theorem}

\begin{proof}
Let $C$ be an arbitrary set from $\bar E^*$. Since $\rho(\theta)\leq \rho_C(\theta)$ for all $\theta\in[0,2\pi]$ and by the assumption there exists a closed interval $[\theta_1,\theta_2]\subseteq [0,2\pi]$ such that
\begin{equation}\label{Eq8}
\rho_{C_0}(\theta)=\rho(\theta)\leq \rho_C(\theta), \ \forall\theta\in[\theta_1,\theta_2].
\end{equation}
In this case, there exists either a $\theta_C\in[\theta_1,\theta_2]$ such that
\begin{equation}\label{Eq9}
\rho_{C_0}(\theta_C)=\rho(\theta_C)<\rho_C(\theta_C)
\end{equation}
or there exists a $\theta_C\in[0,2\pi]\setminus[\theta_1,\theta_2]$ satisfying (\ref{Eq9}). In order to prove this argument, assume that the interval $[\theta_1,\theta_2]$ does not contain any element satisfying (\ref{Eq9}). Therefore it is clear from (\ref{Eq8}) that
\begin{equation}\label{Eq10}
\rho_{C_0}(\theta)=\rho(\theta)= \rho_C(\theta), \ \forall\theta\in[\theta_1,\theta_2].
\end{equation}
If we consider the set $M:=[0,2\pi]\setminus[\theta_1,\theta_2]$ then there exists a $\theta_C\in M$ such that $\rho_{C_0}(\theta_C)<\rho_C(\theta_C)$. Because if there was no such $\theta_C$, in other words if $\rho_{C_0}(\theta)\geq\rho_C(\theta)$ was satisfied for all $\theta\in M$ then with (\ref{Eq10}) we would get
\[\rho_{C_0}(\theta)\geq\rho_C(\theta), \ \forall\theta\in[0,2\pi]\]
that means $C\subseteq C_0$ which contradicts the assumption.
Thus for an arbitrary $C\in E^*\setminus\{C_0\}$ there exists a $\theta_C\in[0,2\pi]$ such that
\begin{equation*}
\rho_{C_0}(\theta_C)<\rho_C(\theta_C)
\end{equation*}
that means $C_0$ can not be discarded from $E^*$ and so $\bar E^*$ is not an upper exhauster of $h$.

Conversely, assume that $\bar E^*$ is not an upper exhauster of $h$. If we set the function
\begin{equation*}
\bar\rho(\theta):=\underset{C\in\bar E^*}{\min}{\underset{v\in C}{\max}{\langle v,u_{\theta}\rangle}}
\end{equation*}
then it is obvious that there exists a $\bar\theta\in[0,2\pi]$ such that
\begin{equation}\label{Eq11}
h(u_{\bar\theta})=\rho(\bar\theta)<\bar\rho(\bar \theta).
\end{equation}
Since $\rho(\bar\theta)=\min\{\bar\rho(\bar\theta),\rho_{C_0}(\bar\theta)\}$ we obtain
\begin{equation*}
\rho_{C_0}(\bar\theta)=\rho(\bar\theta)<\bar\rho(\bar\theta)
\end{equation*}
with the inequality in (\ref{Eq11}). In addition, since the functions $\rho_{C_0}$ and $\bar\rho$ are continuous functions as piecewise sinusoidal curves there exists an interval $[\theta_1,\theta_2]\subset[0,2\pi]$ containing $\bar\theta$ such that
\[\rho_{C_0}(\theta)<\bar\rho(\theta), \ \forall \theta\in [\theta_1,\theta_2].\]
By the definition of function $\bar\rho$ it is clear that $\bar\rho(\theta)\leq\rho_C(\theta)$ on the interval $[\theta_1,\theta_2]$ for all $C\in\bar E^*$. Hence for all $C\in\bar E^*$ we get $\rho_{C_0}<\rho_C$ on $[\theta_1,\theta_2]$ which means that there is no other set of $\bar E^*$ that gives the minimum value to $\rho$ on $[\theta_1,\theta_2]$. In other words, $C_0$ is the unique set that affects the minimum value of $\rho(\theta)$ on this interval. Thus it follows immediately that
\[\rho_{C_0}(\theta)=\rho(\theta), \ \forall \theta\in[\theta_1,\theta_2]\]
which completes the proof.
\end{proof}

\begin{example}\cite[Example 5]{ab1}\label{ex1}
Consider the function $$h(x,y)=\min\{\max\{-x,x,-x+y\};\max\{-2x,-2y,-2x-2y\};\linebreak \max\{3x,-3y,3x-3y\};\max\{4x,4y,4x+4y\}\}.$$ The family $E^*=\{C_1,C_2,C_3,C_4\}$ is an upper exhauster of $h$ where
$$\begin{array}{l}
C_1=conv\{(-1,0),(0,1),(-1,1)\},  C_2=conv\{(-2,0),(0,-2),(-2,-2)\}\\
C_3=conv\{(3,0),(0,-3),(3,-3)\}, C_4=conv\{(4,0),(0,4),(4,4)\} \text{  (see Fig. \ref{fig6}(a))}.
\end{array}$$

As seen in Fig. \ref{fig6}(b) and (c), we deduce from $\theta\rho$-representations of the sets $C_1,\ C_2,\ C_3$ and $C_4$, there exist some intervals for all of these four sets satisfying the equalities

$\rho_{C_1}(\theta)=\rho(\theta),\ \text{for all }\theta\in[0,\frac{\pi}{2}]$

$\rho_{C_2}(\theta)=\rho(\theta),\ \text{for all }\theta\in[\frac{\pi}{2},\pi]$

$\rho_{C_3}(\theta)=\rho(\theta),\ \text{for all }\theta\in[\pi,\frac{3\pi}{2}]$

$\rho_{C_4}(\theta)=\rho(\theta),\ \text{for all }\theta\in[\frac{3\pi}{2},2\pi].$

Therefore none of the sets can be discarded from $E^*$, by Theorem \ref{teo2}.
\end{example}

In the following theorem we present a necessary and sufficient condition for the inclusion-minimality of an upper exhauster. This theorem means that an exhauster of which all elements contribute to the pointwise minimum of the curves $\rho_C$ is minimal by inclusion.

\begin{theorem}\label{teo3}
Let $E^*$ be an upper exhauster of a p.h. function $h$ and assume that elements of $E^*$ do not contain each other. Then, $E^*$ is minimal by inclusion if and only if for all $C\in E^*$ there exists a closed interval $I_C=[\alpha_C,\beta_C]\subset[0,2\pi]$ such that
\begin{equation}\label{Eq11}
\rho_C(\theta)=\rho(\theta),\ \forall\theta\in I_C
\end{equation}
where $\rho(\theta):=\underset{C\in E^*}{\min}{\rho_C(\theta)}$.
\end{theorem}

\begin{proof}
Let $C\in E^*$ be an arbitrary set. Since there isn't any subset of $C$ in $E^*$ and there exists a closed interval $I_C=[\alpha_C,\beta_C]\subset[0,2\pi]$ satisfying (\ref{Eq11}), the set $C$ can not be discarded from $E^*$ by Theorem \ref{teo2}. Since $C$ is arbitrary it follows that there is no other upper exhauster $\widetilde{E}^*$ which is smaller by inclusion than $E^*$ that means $E^*$ is minimal by inclusion.

Conversely, let us assume that $E^*$ is minimal by inclusion and take any $C\in E^*$. Since $\bar E^*=E^*\setminus\{C\}$ is not an upper exhauster of $h$, then by Theorem \ref{teo2} there exists an interval $I_C=[\alpha_C,\beta_C]\subset[0,2\pi]$ such that
\[\rho_C(\theta)=\rho(\theta), \forall\theta\in I_C\]
which is desired result.
\end{proof}

\begin{example}
The upper exhauster given in Example \ref{ex1} is minimal by inclusion by Theorem \ref{teo3}.
%Şekil6
\begin{figure}
\begin{tabular}{clc}
\includegraphics[scale=.5]{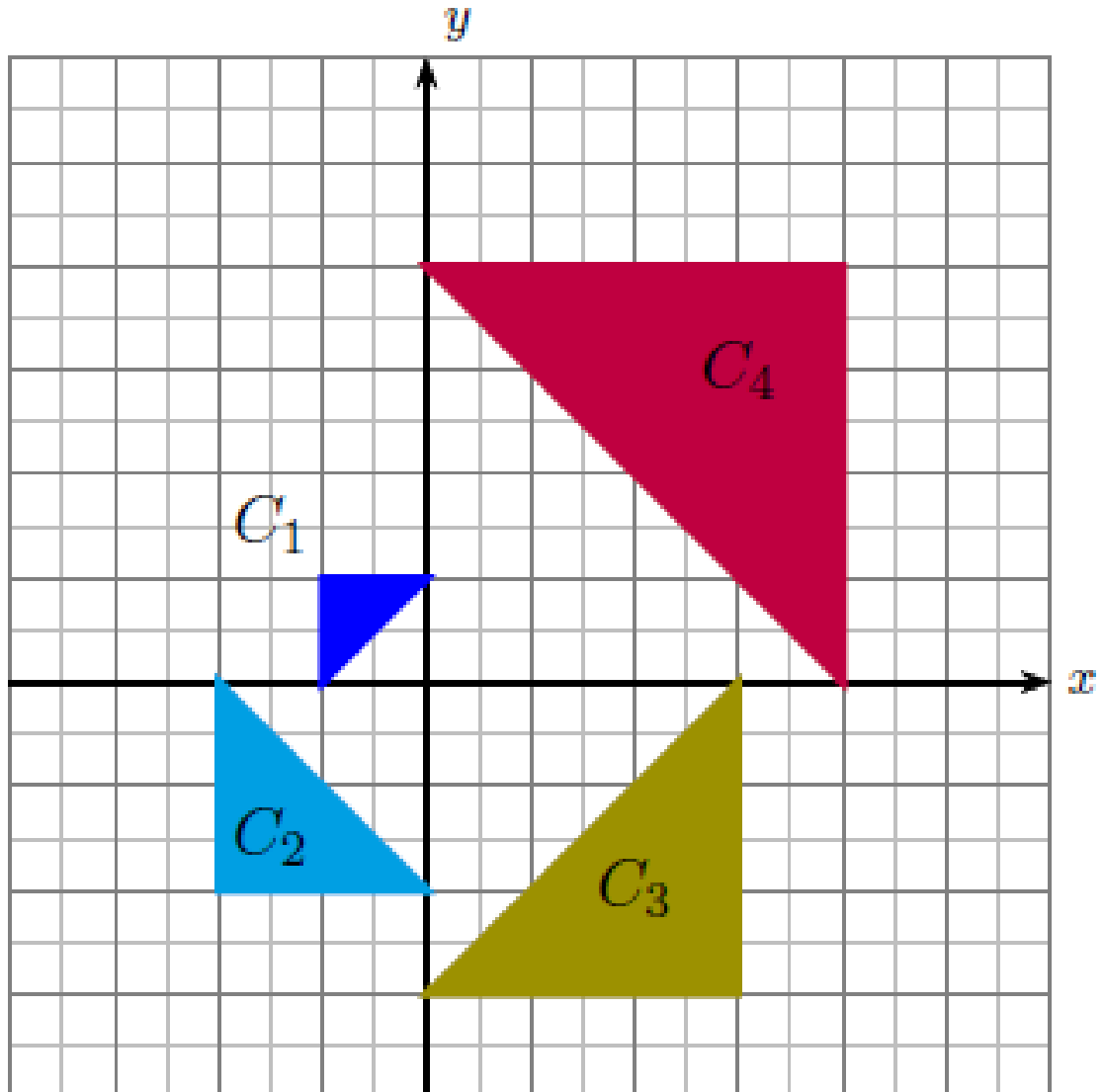}
&
\hspace{-.4cm}
\includegraphics[scale=.5]{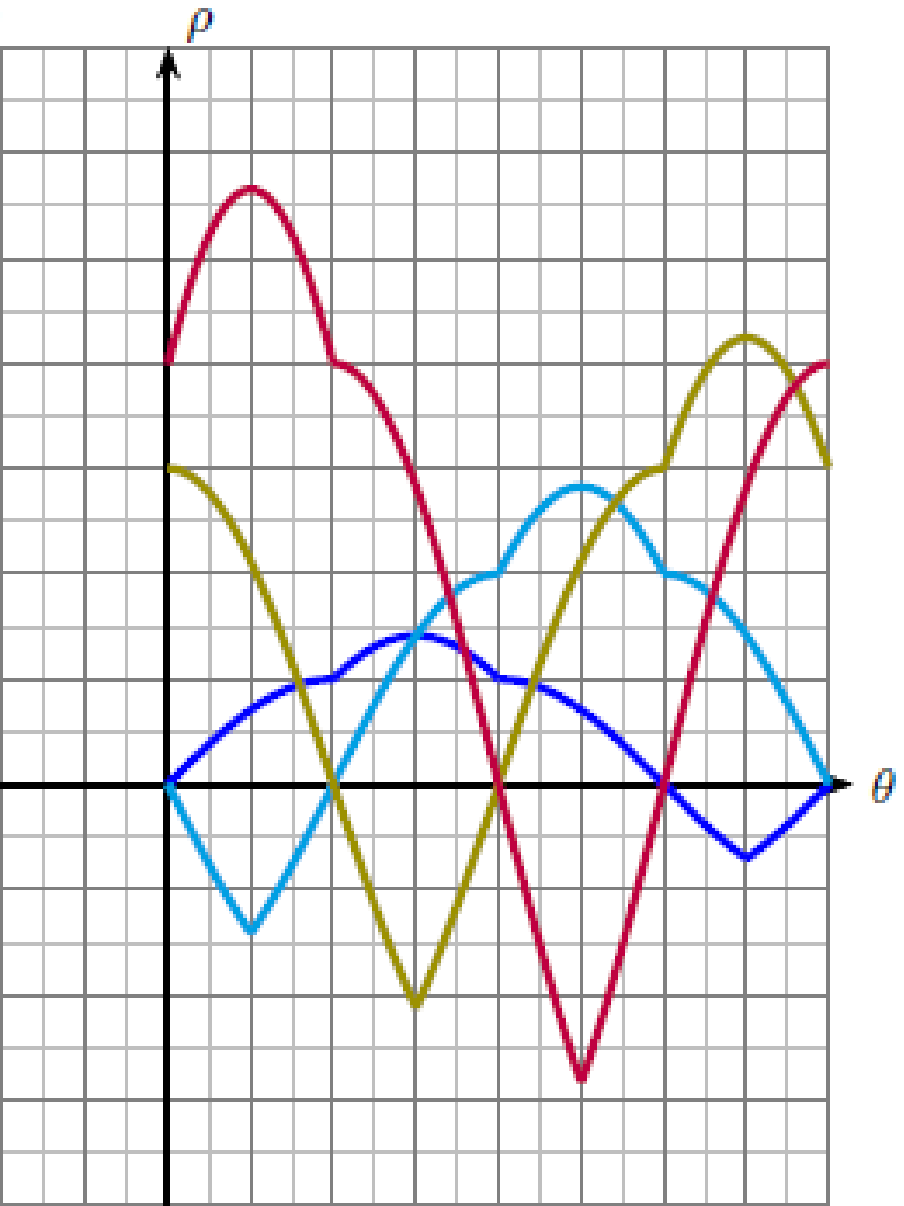}
&
\hspace{-.4cm}
\includegraphics[scale=.5]{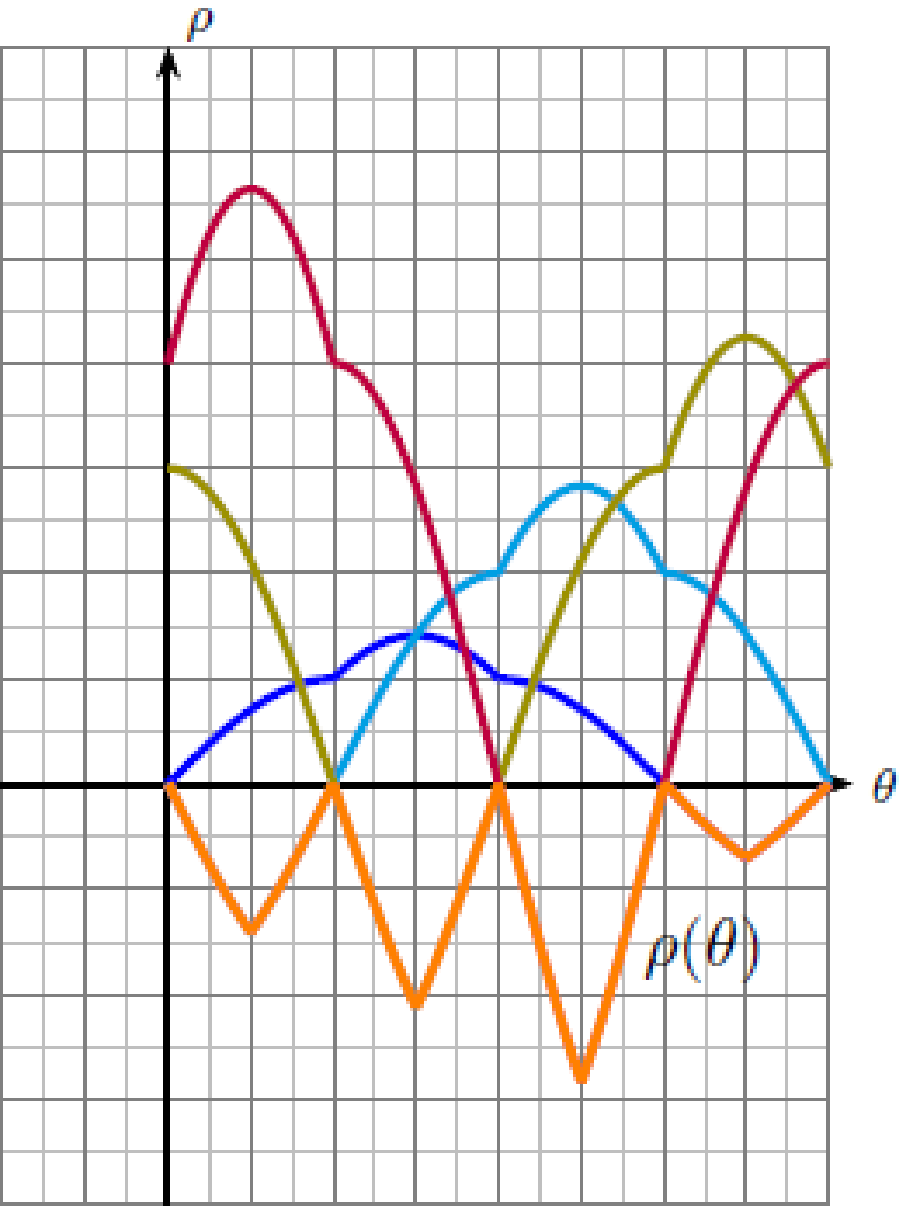}
\end{tabular}
\caption{(a) The sets belonging to the upper exhauster $E^*$ of Example 4. (b) $\theta\rho$-representations of the sets $C_1,C_2,C_3,C_4$ of $E^*$. (c) None of the sets can be discarded from $E^*$ by Theorem \ref{teo2}.}\label{fig6}
\end{figure}
\end{example}

\begin{remark}
It is important that the equation (\ref{Eq11}) is satisfied on an interval, because when (\ref{Eq11}) holds for some discrete points of $[0,2\pi]$ we can not guarantee the inclusion minimality of the exhauster. %The following example shows this fact.
\end{remark}

\begin{example}\cite[Example 3.3]{ros2}
Consider the zero function $h:\R^2\to\R$, $h(x,y)=0$. Define the sets $B_{\alpha}:=B(\alpha,1)$ for $\alpha\in S^1$ which are unit balls tangential to the origin. Then one can see that $$E^*=\{B_{\alpha}\ |\ \alpha\in S^1\}$$ is an upper exhauster of $h$. Let us denote the $\theta\rho$-representations $\rho_{B_{\alpha}}$ by $\rho_{\alpha}$, shortly.

Some of the $\theta\rho$-representations of $B_{\alpha}$ are given in Fig. \ref{fig7}.
%Şekil7
\begin{figure}\vspace{-.6cm}
\includegraphics[width=10cm]{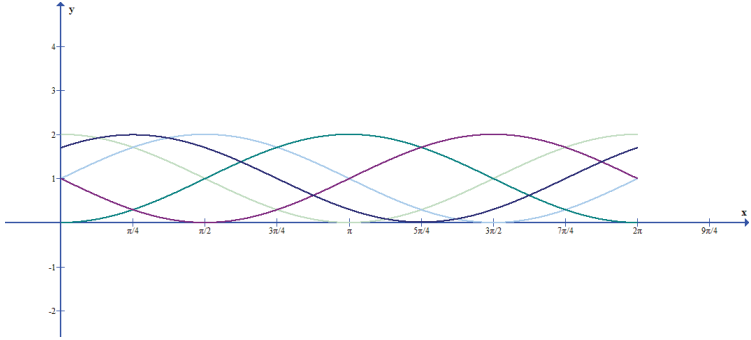}
\caption{$\theta\rho$-representations of some $B_{\alpha}$ belonging to the upper exhauster of zero function.}\label{fig7}
\end{figure}
As seen for all $\alpha\in S^1$ there exists a unique $\theta_{\alpha}\in[0,2\pi]$ such that $\rho_{\alpha}(\theta_{\alpha})=0.$

Moreover, it is clear that $\rho(\theta)=\underset{\alpha\in S^1}{\min}{\rho_{\alpha}(\theta)}=0$. Hence, each of the curves $\rho_{\alpha}$ coincides with the curve $\rho$ at only one point (not on an interval), namely $\theta_{\alpha}$. That is
\[\rho_C(\theta_{\alpha})=\rho(\theta_{\alpha}).\]

On the other hand, $\rho_{\alpha}(\theta)>0=\rho(\theta)$ for all $\theta\in[0,2\pi]\setminus \{\theta_{\alpha}\}$. Therefore, by Corollary \ref{cor1}, we can discard all of the sets $B_{\alpha}$ included in $E^* $ and add the intersection $\underset{\alpha\in S^1}{\bigcap}{B_{\alpha}}=\{(0,0)\}$, since the min operation of support functions means intersection of corresponding sets. Hence, the family $E^*$ can be reduced to a rather smaller exhauster, namely $\widetilde{E}^*=\{\{(0,0)\}\}$, and that means $E^*$ is not minimal by inclusion.
\end{example}

\section{Conclusion}

In this study,  we observe that sets in upper exhausters can be converted into some sinusoidal curves in a very simple and practical way, and when we observe all curves together it can be easily determined which sets are unnecessary in an upper exhauster. These representations of sets in $\theta\rho$-space enable us to develop a very useful method of reducing exhausters. In the continuation of this initial work, it is foreseen to obtain very efficient results in terms of minimality by shape of the exhausters.

It is obvious that results obtained in this study can be established for lower exhausters, analogously.

%\begin{acknowledgements}
%I would like to thank my esteemed professors Mahide K\" \c c\" uk, Yal\c c\i n K\"u\c c\"uk and my dear colleague \. Ilknur Atasever G\" uven\c c for their great support in terms of motivation in this study and to the referees for their valuable comments and contributions.
%\end{acknowledgements}

%Authors must disclose all relationships or interests that
 %could have direct or potential influence or impart bias on
% the work:
%

\section*{Funding}
Not applicable
\section*{Conflict of interest}
The authors declare that they have no conflict of interest.
\section*{Availability of data and material}
Data sharing not applicable to this article as no datasets were generated or analysed during the current study.
\section*{Code availability}
Not applicable
% BibTeX users please use one of
%\bibliographystyle{spbasic}      % basic style, author-year citations
%\bibliographystyle{spmpsci}      % mathematics and physical sciences
%\bibliographystyle{spphys}       % APS-like style for physics
%\bibliography{}   % name your BibTeX data base

\begin{thebibliography}{99}
\small

\bibitem{ab1} Abbasov, M.E., Geometric conditions of reduction of exhausters, J. Global Optim., 74(4), 737-751 (2019)

\bibitem{ab2} Abbasov, M.E., Demyanov, V.F., Proper and adjoint exhausters in nonsmooth analysis: optimality conditions, J. Global Optim., 56(2), 569-585 (2013)

\bibitem{bf} Bonnesen T. and Fenchel W., Theory of Convex Bodies, BCS Associates, Moscow, Idaho (1987)

\bibitem{cas} Castellani, M., A dual representation for proper positively homogeneous functions. J. Glob. Optim. 16(4),
393-400 (2000)

\bibitem{d1} Demyanov, V.F.: Rubinov, A.M., Constructive nonsmooth analysis, Verlag Peter Lang, Frankfurt (1995)

\bibitem{d2} Demyanov, V.F.: Exhausters of a positively homogeneous function, Optimization, 45, 13-29 (1999)

\bibitem{dr1} Demyanov, V.F., Roshchina, V.A.: Constrained optimality conditions in terms of upper and lower exhausters, Appl. Comput. Math., 4(2), 25-35 (2005)

\bibitem{dr2} Demyanov, V.F., Roshchina, V.A.: Optimality conditions in terms of upper and lower exhausters, Optimization, 55(5-6), 525-540 (2006)

\bibitem{dr3} Demyanov, V.F., Roshchina, V.A.: Exhausters and subdifferentials in non-smooth analysis, Optimization, 57(1), 41-56 (2008)

\bibitem{dr4} Demyanov, V.F., Roshchina, V.A: Generalized Subdifferentials and Exhausters in Nonsmooth Analysis, Doklady Mathematics, 76(2), 652-655 (2007)

\bibitem{gosh} Gosh P.K., Kumar K.V.: Support function representation of convex bodies, its application in geometric computing, and some related representations, Computer Vision and Image Understanding, 72(3), 379-403 (1998)

\bibitem{gfinite} Grzybowski J., Pallaschke D., Urbanski R.: Reduction of Finite exhausters, J. Global Optim. 46(4), 589-601 (2010)

\bibitem{weakexh} K\"u\c c\"uk, M., Urbanski, R., Grzybowski J., K\"u\c c\"uk, Y., Atasever G\"uven\c c, \.I., Tozkan, D., Soyertem, M.: Weak subdifferential/superdifferential, weak exhausters and optimality conditions, Optimization, 64(10), 2199-2212 (2015)

\bibitem{reduction} K\"u\c c\"uk, M., Urbanski, R., Grzybowski J., K\"u\c c\"uk, Y., Atasever G\" uven\c c, \.I., Tozkan, D., Soyertem, M.: Reduction of Weak Exhausters and Optimality Conditions via Reduced Weak Exhausters, J. Optim. Theory Appl., 165, 693-707 (2015)

\bibitem{mink} Minkowski, H.: Volumen und Oberfl\"{a}che, Math. Ann., 57(4), 447-495 (1903)

\bibitem{psch} Pschenichnyi, B.N.: Convex Analysis and Extremal Problems. Nauka Publishers, Moscow (1980)

\bibitem{ros1} Roshchina, V.A.: Relationships between upper exhausters and the basic subdifferential in variational analysis, J. Math. Anal. Appl., 33, 261-272 (2007)

\bibitem{ros2} Roshchina, V.A.: Reducing Exhausters, J. Optim. Theory Appl., 136, 261-273 (2008)

\bibitem{ros3} Roshchina, V.A.: On conditions for minimality of exhausters. J. Convex Anal. 15(4), 859-868 (2008)

\bibitem{rock} Rockafellar R.T.: Convex analysis. Princeton (NJ): Princeton University Press  (1970)

\bibitem{schn} Schneider R.: Convex Bodies: The Brunn-Minkowski Theory, Cambridge Univ. Press, Cambridge, UK (1993)

\end{thebibliography}

% Non-BibTeX users please use

%
% and use \bibitem to create references. Consult the Instructions
% for authors for reference list style.

%\bibitem{RefJ}
% Format for Journal Reference
%Author, Article title, Journal, Volume, page numbers (year)
% Format for books
%\bibitem{RefB}
%Author, Book title, page numbers. Publisher, place (year)
% etc

\end{document}